\documentclass[12pt]{amsart}
\usepackage[T1]{fontenc}
\usepackage[english]{babel}
\usepackage{amscd,amsmath,amsthm,amssymb,graphics}
\usepackage{lmodern,pst-node}
\usepackage{pstcol,pst-plot,pst-3d}
\usepackage{multicol}
\usepackage{epic,eepic}
\usepackage{amsfonts,amssymb,amscd,amsmath,enumitem,verbatim}
\psset{unit=0.7cm,linewidth=0.8pt,arrowsize=2.5pt 4}

\newpsstyle{fatline}{linewidth=1.5pt}
\newpsstyle{fyp}{fillstyle=solid,fillcolor=verylight}
\definecolor{verylight}{gray}{0.97}
\definecolor{light}{gray}{0.9}
\definecolor{medium}{gray}{0.85}
\definecolor{dark}{gray}{0.6}

\def\frk{\frak}               

\def\mm{{\frk m}}

\def\Phi{{\frk n}}
\def\Phi{{\frk N}}
\def\xb{{\bold x}}

\def\opn#1#2{\def#1{\operatorname{#2}}} 
\opn\chara{char} \opn\length{\ell} \opn\pd{pd} \opn\rk{rk}
\opn\projdim{proj\,dim} \opn\injdim{inj\,dim} \opn\rank{rank}
\opn\depth{depth} \opn\grade{grade} \opn\height{height}
\opn\embdim{emb\,dim} \opn\codim{codim}

\opn\Tr{Tr} \opn\bigrank{big\,rank}
\opn\superheight{superheight}\opn\lcm{lcm}
\opn\trdeg{tr\,deg}
\opn\reg{reg} \opn\lreg{lreg} \opn\ini{in} \opn\lpd{lpd}
\opn\size{size}\opn\bigsize{bigsize}
\opn\cosize{cosize}\opn\bigcosize{bigcosize}
\opn\sdepth{sdepth}\opn\sreg{sreg}
\opn\link{link}\opn\fdepth{fdepth}
\opn\div{div} \opn\Div{Div} \opn\cl{cl} \opn\Cl{Cl}
\let\epsilon\varepsilon
\let\phi=\varphi
\let\kappa=\varkappa
\opn\Spec{Spec} \opn\Supp{Supp} \opn\supp{supp} \opn\Sing{Sing}
\opn\Ass{Ass} \opn\Min{Min}\opn\Mon{Mon} \opn\dstab{dstab} \opn\astab{astab}
\opn\Syz{Syz}
\opn\Ann{Ann} \opn\Rad{Rad} \opn\Soc{Soc}
\opn\Im{Im} \opn\Ker{Ker} \opn\Coker{Coker} \opn\Am{Am}
\opn\Hom{Hom} \opn\Tor{Tor} \opn\Ext{Ext} \opn\End{End}
\opn\Aut{Aut} \opn\id{id}

\opn\nat{nat}
\opn\pff{pf}
\opn\Pf{Pf} \opn\GL{GL} \opn\SL{SL} \opn\mod{mod} \opn\ord{ord}
\opn\Gin{Gin} \opn\Hilb{Hilb}\opn\sort{sort}
\opn\initial{init}
\opn\ende{end}
\opn\height{height}
\opn\type{type}
\opn\aff{aff} \opn\con{conv} \opn\relint{relint} \opn\st{st}
\opn\lk{lk} \opn\cn{cn} \opn\core{core} \opn\vol{vol}
\opn\link{link} \opn\star{star}\opn\lex{lex}
\opn\gr{gr}

\def\pot#1#2{#1[\kern-0.28ex[#2]\kern-0.28ex]}

\opn\dirlim{\underrightarrow{\lim}}
\opn\inivlim{\underleftarrow{\lim}}
\let\union=\cup
\let\sect=\cap

\let\iso=\cong

\let\Sect=\bigcap

\let\to=\rightarrow

\def\Implies{\ifmmode\Longrightarrow \else
        \unskip${}\Longrightarrow{}$\ignorespaces\fi}
\def\implies{\ifmmode\Rightarrow \else
        \unskip${}\Rightarrow{}$\ignorespaces\fi}
\def\iff{\ifmmode\Longleftrightarrow \else
        \unskip${}\Longleftrightarrow{}$\ignorespaces\fi}

\let\:=\colon
 \theoremstyle{plain}
\newtheorem{Theorem}{Theorem}[section]
 
 \newtheorem{Corollary}[Theorem]{Corollary}
 \newtheorem{Proposition}[Theorem]{Proposition}

 \theoremstyle{definition}

 \newtheorem{Example}[Theorem]{Example}

\let\epsilon\varepsilon
\let\kappa=\varkappa
\def\qed{\ifhmode\textqed\fi
      \ifmmode\ifinner\quad\qedsymbol\else\dispqed\fi\fi}
\def\textqed{\unskip\nobreak\penalty50
       \hskip2em\hbox{}\nobreak\hfil\qedsymbol
       \parfillskip=0pt \finalhyphendemerits=0}
\def\dispqed{\rlap{\qquad\qedsymbol}}

\opn\dis{dis}
\def\pnt{{\raise0.5mm\hbox{\large\bf.}}}

\opn\Lex{Lex}

\begin{document}
\title{Bi-Cohen-Macaulay graphs}
\author {J\"urgen Herzog and Ahad Rahimi}

\address{J\"urgen Herzog, Fachbereich Mathematik, Universit\"at Duisburg-Essen, Fakult\"at f\"ur Mathematik, 45117
Essen, Germany}

\email{juergen.herzog@uni-essen.de}
\address{ Ahad Rahimi, Department of Mathematics, Razi University, Kermanshah,
 Iran}
 \email{ahad.rahimi@razi.ac.ir}

\begin{abstract}
In this paper we consider bi-Cohen-Macaulay graphs, and give a complete classification of such graphs in the case they are bipartite or chordal. General bi-Cohen-Macaulay graphs are classified up to separation. The inseparable bi-Cohen-Macaulay graphs are  determined.  We establish  a bijection between the set of all trees and the set of inseparable bi-Cohen-Macaulay graphs.
\end{abstract}

\thanks{This paper was written during the visit of the second author at Universit\"at Duisburg-Essen, Campus Essen. He is grateful for its hospitality. }

\maketitle

\section*{Introduction}

A simplicial complex $\Delta$ is called {\em bi-Cohen-Macaulay} (bi-CM), if $\Delta$ and its Alexander dual $\Delta^\vee$ are Cohen-Macaulay. This concept was introduced by Fl{\o}ystad and Vatne in \cite{FV}. In that paper the authors associated to each simplicial complex $\Delta$ in a natural way a complex of coherent sheaves and showed that this complex reduces to a coherent sheaf if and only if $\Delta$ is bi-CM.

The present  paper is an attempt to classify all  bi-CM graphs. Given a field $K$ and a simple graph on the vertex set $[n]=\{1,2,\ldots,n\}$,  one associates with $G$ the edge ideal $I_G$ of $G$, whose generators are the monomials $x_ix_j$ with $\{i,j\}$ an edge of $G$. We say that $G$ is bi-CM if the simplicial complex whose Stanley-Reisner ideal coincides with $I_G$ is bi-CM. Actually, this simplicial complex is the so-called {\em independence complex} of $G$. Its   faces are the independent sets of $G$, that is, subsets $D$ of $[n]$ with $\{i,j\}\not\subset D$ for all edges $\{i,j\}$ of $G$.

By its very definition, any bi-CM graph  is also a Cohen-Macaulay graph (CM graph). A complete classification of all CM graphs is hopeless if not impossible. However, such a classification is given for bipartite graphs \cite[Theorem 3.4]{HH1} and for chordal graphs \cite{HHZ}. We refer the reader to the books \cite{HH} and  \cite{V} for a good survey on edge ideals and its algebraic and homological properties.

Based on the classification of bipartite and chordal CM graphs,  we provide in Section~\ref{class} a classification  of bipartite and chordal bi-CM graphs, see Theorem~\ref{bipartite1} and Theorem~\ref{chordal}. In Section~\ref{properties} we first present   various characterizations of bi-CM graphs. By using the Eagon-Reiner theorem \cite{ER}, one notices that the graph $G$ is bi-CM if and only if it is CM and $I_G$ has a linear resolution. Cohen-Macaulay ideals generated in degree 2 with linear resolution are of very special nature. They all arise as deformations of the square of the maximal ideal of a suitable polynomial ring. From this fact arise  constraints on the number of edges of the graph and on the Betti numbers of $I_G$.

Though a complete classification of all bi-CM  graphs seems to be again impossible, a classification of all bi-CM graphs up to separation can be given, and this is the subject of the remaining sections.

A {\em separation} of the graph $G$ with respect to the vertex $i$ is  a graph $G'$ whose vertex set is $[n]\union\{i'\}$ having the property that $G$ is obtained from $G'$ by identifying $i$ with $i'$ and such that $x_i-x_{i'}$ is a non-zerodivisor modulo $I_{G'}$. The algebraic condition on separation makes sure that the essential algebraic and homological invariants of $I_G$ and $I_{G'}$ are the same. In particular, $G$ is bi-CM if and only if $G'$ is bi-CM.  A graph which does not allow any separation is called {\em inseparable}, and a inseparable graph which is obtained by  a finite number of separation steps from $G$ is called a {\em separable model} of $G$. Any graph admits separable models and the number of separable models of a graph is finite. Separable and inseparable graphs from the view point of deformation theory have been studied in \cite{ABHL}.

In Section~\ref{main} we determine all inseparable bi-CM graphs on $[n]$ vertices. Indeed, in Theorem~\ref{maintheorem} it is shown that for any tree $T$ on the vertex set $[n]$ there exists a unique inseparable bi-CM graph $G_T$ determined by $T$, and any inseparable bi-CM graph is of this form. Furthermore, if $G$ is an arbitrary bi-CM graph and $T$ is the relation graph of the Alexander dual of $I_G$,  then $G_T$ is a separable model of $G$.

For a bi-CM graph $G$, the Alexander dual $J=(I_G)^\vee$ of $I_G$ is a Cohen-Macaulay ideal of codimension 2 with linear resolution. As described in \cite{BH1}, one attaches to any relation matrix of $J$ a relation tree $T$. Replacing the entries in this matrix by distinct variables with the same sign, one obtains the  so-called {\em generic relation matrix} whose ideals of $2$-minors $J_T$ and its Alexander has been computed in \cite{N}. This theory is described in Section~\ref{generic}. The Alexander dual of $J_T$ is the edge ideal of  graph, which actually is the graph $G_T$ mentioned before and which serves as a separable model of $G$.

\section{Preliminaries and various characterizations  of Bi-Cohen-Macaulay graphs}
\label{properties}

In this section we recall some of the standard notions of graph theory which are relevant for this paper, introduce the bi-CM graphs and present various equivalent conditions of a graph to be bi-CM.

The graphs considered here will all be finite, simple graphs, that is, they will have no double edges and no loops. Furthermore we assume that $G$ has no isolated vertices. The vertex set of $G$ will be denoted $V(G)$ and will be the set $[n]=\{1,2,\ldots,n\}$, unless otherwise stated. The set of edges of $G$ we denote by $E(G)$.

A subset $F\subset [n]$ is called a {\em clique} of $G$, if $\{i,j\}\in E(G)$ for all $i,j\in F$ with $i\neq j$. The set of all cliques of $G$ is a simplicial complex, denoted $\Delta(G)$.

A subset $C\subset [n]$ is called a {\em vertex cover} of $G$  if $C\sect\{i,j\}\neq \emptyset$ for all edges $\{i,j\}$ of $G$. The graph $G$ is called {\em unmixed} if all minimal vertex covers of $G$ have the same cardinality. This concept has an algebraic counterpart. We fix a field $K$ and consider the ideal $I_G\subset S=K[x_1,\ldots,x_n]$ which is generated by all monomials $x_ix_j$ with $\{i,j\}\in E(G)$. The ideal $I_G$ is called the {\em edge ideal} of $G$.  Let $C\subset [n]$.  Then the monomial prime ideal $P_C=(\{x_i\: i\in C\})$ is a minimal prime ideal of $I_G$ if and only if $C$ is a minimal vertex cover of $G$. Thus $G$ is unmixed if and only if $I_G$ is unmixed in the algebraic sense. A subset $D\subset [n]$ is called an {\em independent set} of $G$ if $D$ contains no set $\{i,j\}$ which is an edge of $G$. Note that $D$ is an independent set of $G$ if and only if $[n]\setminus D$ is a vertex cover. Thus the minimal vertex covers of $G$ correspond to the maximal independent sets of $G$. The cardinality of a maximal independent is called the {\em independence number} of $G$. It follows that the Krull dimension of $S/I_G$ is equal to $c$, where $c$ is the independence number of $G$.

The graph $G$ is called {\em bipartite} if $V(G)$ is  the disjoint union of $V_1$ and $V_2$ such that $V_1$ and $V_2$ are independent sets, and $G$ is called {\em disconnected} if $V(G)$ is the disjoint union of $W_1$ and $W_2$ and there is no edge $\{i,j\}$ of $G$ with $i\in W_1$ and $j\in W_2$. The graph $G$ is called {\em connected} if it is not disconnected.

A {\em cycle} $C$ (of length $r$) in $G$  is a sequence of edges $\{i_k,j_k\}$ with $k=1,2,\ldots,r$ such that $j_k=i_{k+1}$ for $k=1,\ldots,r-1$ and $j_r=i_1$. A cord of $C$ is an edge $\{i,j\}$  of $G$ with $i,j\in\{i_1,\ldots,i_r\}$ and $\{i,j\}$ is not an edge of $C$. The graph $G$ is called {\em chordal} if each cycle of $G$  of length $\geq 4$ has a chord. A graph which has no cycle and which is connected is called a {\em tree}.

\medskip
Now we recall the main concept we are dealing with in this paper. Let $I\subset S$ be a squarefree monomial ideal. Then $I=\Sect_{j=1}^m P_{j}$ where each of the $P_j$  is a monomial prime ideal of $I$.  The ideal $I^\vee$ which is minimally generated by the monomials $u_{j}=\prod_{x_i\in P_j}x_i$ is called the {\em Alexander dual} of $I$. One has $(I^\vee)^ \vee=I$. In the case that $I=I_G$,  each  $P_j$ is generated by the variables  corresponding to a minimal vertex cover of $G$. Therefore, $(I_G)^\vee$ is also called the {\em vertex cover ideal} of $G$.

According to \cite{FV} a squarefree monomial ideal $I\subset S$  is called {\em bi-Cohen-Macaulay} (or simply bi-CM) if  $I$ as well as  the Alexander dual $I^\vee$ of $I$ is a Cohen-Macaulay ideal. A graph $G$ is called {\em Cohen-Macaulay} or {\em bi-Cohen-Macaulay (over $K$)} (CM or bi-CM for short),  if $I_G$ is CM or bi-CM. One important result  regarding  the Alexander dual that will be used frequently in this paper is the Eagon-Reiner theorem which says that $I$ is a Cohen-Macaulay ideal if and only if $I^\vee$ has a linear resolution.
Thus the Eagon-Reiner theorem implies  that $I$ is bi-CM if and only if $I$ is a Cohen-Macaulay ideal with  linear resolution. From this description it follows that a bi-CM graph is connected. Indeed, if this is not the case, then there are induced subgraphs $G_1, G_2\subset G$ such that $V(G)$ is the disjoint union of $V(G_1)$ and $V(G_2)$. It follows that $I_G=I_{G_1}+I_{G_2}$,  and the ideals $I_{G_1}$ and $I_{G_2}$ are ideals in a different set of variables. Therefore, the free resolution of $S/I_G$ is obtained as the tensor product of the resolutions of $S/I_{G_1}$ and $S/I_{G_2}$. This implies that $I_G$ has relations of degree 4, so that $I_G$ does not have a linear resolution.

\medskip
From now on we will always assume that $G$ is connected, without further mentioning it.

\begin{Proposition}
\label{basic}
Let $K$ be an infinite field and  $G$  a graph on the vertex set $[n]$ with independence number $c$. The following conditions are equivalent:
\begin{enumerate}
\item[{\em (a)}] $G$ is a bi-CM graph over $K$;
\item[{\em (b)}] $G$ is a CM graph over $K$, and $S/I_G$ modulo a maximal regular sequence of linear forms is isomorphic to $T/\mm_T^2$ where $T$ is the polynomial ring over $K$ in $n-c$ variables and $\mm_T$ is the graded maximal ideal of $T$.
\end{enumerate}
\end{Proposition}

\begin{proof} We only need to show that $I_G$ has a linear resolution if and only if condition (b) holds. Since $K$ is infinite and since $S/I_G$ is Cohen-Macaulay of dimension $c$, there exists  a  regular sequence $\xb$ of linear forms on $S/I_G$ of length $c$.  Let $T=S/(\xb)$. Then $T$ is isomorphic to  a polynomial ring in $n-c$ variables. Let $J$ be the image of $I_G$ in $T$. Then $J$ is generated  in degree $2$ and has a linear resolution if and only if $I_G$ has linear resolution. Moreover, $J$ is $\mm_T$-primary.  The only $\mm_T$-primary ideals with linear resolution are the powers of $\mm_T$. Thus, $I_G$ has a linear resolution if and only if $J=\mm_T^2$.
\end{proof}

\begin{Corollary}
\label{edges}
Let $G$ be  a graph on the vertex set $[n]$ with independence number $c$. The following conditions are equivalent:
\begin{enumerate}
\item[{\em (a)}] $G$ is a bi-CM graph over $K$;
\item[{\em (b)}] $G$ is a CM graph over $K$  and $|E(G)|={n-c+1\choose  2}$;
\item[{\em (c)}] $G$ is a CM graph over $K$ and  the number of minimal vertex covers of $G$ is equal to $n-c+1$;
\item[{\em (d)}] $\beta_i(I_G)=(i+1){n-c+1\choose i+2}$ for $i=0,\ldots, n-c-1$.
\end{enumerate}
\end{Corollary}

\begin{proof} For the proof of the equivalent conditions we may assume that $K$ is infinite and hence we may use Proposition~\ref{basic}.

(a)\iff (b): With the notation of Proposition~\ref{basic} we have $J=\mm_T^2$ if and only if the number of generators of $J$ is equal to ${n-c+1\choose  2}$. Since $I_G$ and $J$ have the same number of generators and since the number of generators of $I_G$ is equal to $|E(G)|$, the assertion follows.

(b)\iff (c): Since $S/I_G$ is Cohen-Macaulay, the multiplicity of $S/I_G$ is equal  to  the length $\length(T/J)$ of $T/J$. On the other hand, the multiplicity is also the number of minimal prime ideals of $I_G$ which coincides with the number of minimal vertex covers of $G$. Thus the length of $T/J$ is equal to the number of minimal vertex covers of $G$. Since  $J=\mm_T^2$ if and only if $\length(T/J)=n-c+1$, the assertion follows.

(a)\implies (d): Note that  $\beta_i(I_G)=\beta_i(J)$ for all $i$. Since $J$ is isomorphic to the ideal of $2$-minors of the matrix
\[
\begin{pmatrix}
 y_{1} & y_2 & \ldots   & y_{n-c}   & 0 \\
  0    & y_1 & \ldots   & y_{n-c-1} & y_{n-c}
\end{pmatrix}
\]
in the variables $y_1,\ldots,y_{n-c}$, the Eagon-Northcott complex (\cite{BV}, \cite{E}) provides a free resolution of $J$ and the desired result follows.

(d)\implies (a): It follows from the description of the Betti numbers of $I_G$ that $\projdim S/I_G=n-c$. Thus, $\depth S/I_G=c$.  Since $\dim S/I_G=c$, it follows that $I_G$ is a Cohen-Macaulay ideal. Since $|E(G)|=\beta_0(I_G)={n-c+1\choose 2}$, condition (b) is satisfied, and hence $G$ is bi-CM, as desired.
\end{proof}

Finally we note that $G$ is a bi-CM graph over $K$ if and only if the vertex cover ideal of $G$ is a codimension $2$ Cohen-Macaulay ideal with linear relations. Indeed,  let $J_G$ be the vertex cover ideal of $G$. Since $J_G=(I_G)^\vee$, it follows from the Eagon-Reiner theorem $J_G$ is bi-CM if and only if $I_G$ is bi-CM.

\section{The classification of  bipartite and chordal bi-CM graphs}
\label{class}

In this section we give a full classification of  the bipartite and chordal bi-CM graphs.

\begin{Theorem}
\label{bipartite1}
Let $G$ be a bipartite graph on the vertex set $V$ with bipartition $V=V_1\union V_2$ where  $V_1=\{v_1,\ldots,v_n\}$ and $V_2=\{w_1,\ldots,w_m\}$. Then the following conditions are equivalent:
\begin{enumerate}
\item[{\em (a)}] $G$ is a bi-CM graph;
\item[{\em (b)}] $n=m$ and $E(G)=\{\{v_i,w_j\}\: 1\leq i\leq j\leq n\}$.
\end{enumerate}
\end{Theorem}

\begin{proof}
(a)\implies (b): Since $G$ is a bi-CM graph, it is in particular a CM-graph, and so  $n=m$, and by \cite[Theorem 9.1.13]{HH} there exists a poset $P = \{p_1,\dots, p_n\}$ such that $G=G(P)$. Here  $G(P)$ is the bipartite graph on $V=\{v_1,\ldots,v_n, w_1,\ldots,w_n\}$ whose edges are those $2$-element subset $\{v_i, w_j\}$ of $V$ such that $p_i\leq p_j$. Thus $I_G=I_{G(P)}= H_P^\vee$,  where
\[
H_P=\bigcap_{p_i\leq p_j}(x_i, y_j)
\]
is an ideal of $S=K[\{x_i,y_i\}_{p_i\in P}]$,  the polynomial ring in $2n$ variables over $K$.
Since $G$ is bi-CM, it follows that $H_P$ is Cohen--Macaulay, and hence
\[
\projdim S/H_P=2n-\depth S/H_P=2n-\dim S/H_P=\height H_P=2.
\]
Thus $\projdim H_P=1$,  and hence, by \cite[Corollary 2.2]{HH1}, the Sperner number of $P$, i.e., the maximum of the cardinalities of antichains of $P$ equals $1$. This implies that $P$ is a chain, and this yields (b).

(b)\implies (a): The graph $G$ described in (b) is of the form $G=G(P)$ where $P$ is a chain. By what is said in (a)\implies (b), it follows that $G$ is bi-CM.
\end{proof}

The following picture shows a bi-CM bipartite graph for $n=4$.

\begin{figure}[hbt]
\begin{center}
\psset{unit=0.8cm}
\begin{pspicture}(-0.2,-2)(2,2)
\rput(-0.8,-1){

\rput(0,0){$\bullet$}
\rput(0,1.5){$\bullet$}
\rput(1.5,0){$\bullet$}
\rput(1.5,1.5){$\bullet$}
\rput(3, 0){$\bullet$}
\rput(3, 1.5){$\bullet$}
\rput(4.5, 0){$\bullet$}
\rput(4.5, 1.5){$\bullet$}

\psline(0,.0)(0, 1.5)
\psline(1.5,.0)(1.5,1.5)
\psline(3, .0)(3, 1.5)
\psline(4.5, .0)(4.5,1.5)

\psline(0.0, 1.5)(1.5,.0)
\psline(0.0, 1.55)(3,.0)
\psline(0.03, 1.5)(4.5,.0)

\psline(1.5, 1.5)(3,.0)
\psline(1.5, 1.5)(4.5,.0)

\psline(3,1.5)(4.5,.0)

\rput(0,2){$x_{1}$}
\rput(1.5,2){$x_{2}$}
\rput(3, 2){$x_{3}$}
\rput(4.5,2){$x_{4}$}

\rput(0,-.5){$y_{1}$}
\rput(1.5,-.5){$y_{2}$}
\rput(3, -.5){$y_{3}$}
\rput(4.5,-.5){$y_{4}$}

}
\end{pspicture}
\end{center}
\caption{A bi-CM bipartite graph.}\label{bipartite}
\end{figure}
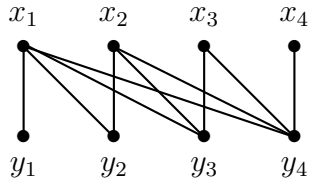

\begin{Theorem}
\label{chordal}
Let $G$ be a chordal graph on the vertex set $[n]$. The following conditions are equivalent:
\begin{enumerate}
\item[{\em (a)}] $G$ is a bi-CM graph;
\item[{\em (b)}] Let $F_1,\ldots, F_m$ be the facets of the clique complex of $G$. Then $m=1$, \\ or $m>1$ and
\begin{enumerate}
\item[{\em (i)}] $V(G)=V(F_1)\union V(F_2)\union \ldots \union V(F_m)$, and this union is disjoint;
\item[{\em (ii)}] each $F_i$ has exactly one free vertex $j_i$;
\item[{\em (iii)}] the restriction of $G$ to $[n]\setminus\{j_1,\ldots,j_m\}$ is a clique.
\end{enumerate}
\end{enumerate}
\end{Theorem}

\begin{proof} Let $I_{n,d}$ be the ideal generated by all squarefree monomials of degree $d$ in $S=K[x_1,\ldots,x_n]$. It is known (and easy to prove) that $I_{n,d}^\vee=I_{n,n-d+1}$, and that all these ideals are Cohen-Macaulay, and hence all bi-CM. If  $m=1$, then $I_G=I_{n,2}$ and the result follows.

Now let $m>1$.  A bi-CM graph is a CM graph. The CM  chordal graphs have been classified in \cite{HHZ}: they are the chordal graphs satisfying (b)(i). Thus for the proof of the theorem we may assume that (b)(i) holds and simply have to show that (b)(ii) and (b)(iii) are satisfied if and only if $I_G$ has a linear resolution.

Let $P_i$ be the monomial prime ideal generated by the variables $x_k$ with  $k\in V(F_i)\setminus \{j_i\}$, and let $G'$ be subgraph of $G$ whose edges do not belong to any $F_i$. It is shown in the proof of \cite[Corollary 2.1]{HHZ} that there exists a regular sequence on $S/I_G$ such that after reduction modulo this sequence one obtains the ideal $J\subset T$ where $T$ is the polynomial ring on the variables $x_k$ with $k\neq j_i$ for $i=1,\ldots,m$ and where
\begin{eqnarray}
\label{reduction}
J=(P_1^2,\ldots, P_m^2, I_{G'}).
\end{eqnarray}
By  Proposition~\ref{basic},  it follows that $I_G$ has a linear resolution if and only if $J=\mm_T^2$, where $\mm_T$ denotes the graded maximal ideal of $T$.

So, now suppose first that $I_G$ has a linear resolution, and hence $J=\mm_T^2$. Suppose that some  $F_i$ has more than one free vertex, say $F_i$ has the vertex $k$ with $k\neq j_i$. Choose any $F_t$ different from $F_i$ and let $l\in F_j$ with $l\neq j_t$. Then $x_k$ and $x_l$ belong to $T$ but $x_kx_l\not\in J$ as can be seen from (\ref{reduction}). This is a contradiction. Thus (b)(ii) follows.

Suppose next that the graph $G''$ which is the restriction of $G$ to $[n]\setminus\{j_1,\ldots,j_m\}$ is not a clique.  Then there exist $i,j\in V(G'')$ such that $\{i,j\}\not\in E(G'')$. However, since all $x_k$ with $k\in V(G'')$ belong to $T$ and since $J=\mm_T^2$, it follows $x_ix_j\in J$. Thus, by (\ref{reduction}),  $x_ix_j \in P_k^2$ for some $k$ or $x_ix_j\in I_{G'}$. Since (b)(ii) holds, this implies  in both cases that $\{i,j\}\in E(G'')$, a contradiction.
Thus (b)(iii) follows.

Conversely, suppose (b)(ii) and (b)(iii) hold. We want to show that $J=\mm_T^2$. Let $x_i,x_j\in T$. We have  to show that $x_ix_j\in J$. It follows from the description of $J$ that $x_k^2\in J$ for all $x_k\in T$. Thus we may assume that $i\neq j$. If $\{i,j\}$ is not an edge of any $F_k$, then by definition it is an edge of $G'$, and hence $x_ix_j\in I_{G'}\subset J$. On the other hand, if $ \{i,j\}$ is an edge of  $F_k$  for some $k$, then $i,j\neq i_k$, and hence $x_ix_j\in P_k^2\subset J$. Thus the  desired conclusion follows.
\end{proof}

Let $G$ be a chordal bi-CM graph as in Theorem~\ref{chordal}(b) with $m>1$.  We call the complete graph $G''$ which is the restriction of $G$ to $[n]\setminus\{j_1,\ldots,j_m\}$ the {\em center} of $G$.

The following picture shows, up to isomorphism,  all bi-CM chordal graphs whose center is the complete graph $K_4$ on $4$ vertices:

\begin{figure}[hbt]
\begin{center}
\psset{unit=0.7cm}
\begin{pspicture}(0, -1)(2,1)
\rput(0.3,-1){
\rput(0,0){$\bullet$}
\rput(1.5,0){$\bullet$}
\rput(1.5,1.5){$\bullet$}
\rput(3, 1.5){$\bullet$}
\rput(3, 0){$\bullet$}
\rput(-1.5, 1.5){$\bullet$}
\rput(0, 1.5){$\bullet$}
\rput(-1.5, 0){$\bullet$}

\psline(0,0)(0,1.5)
\psline(0,0)(1.5,0)
\psline(1.5,0)(1.5,1.5)
\psline(0, 1.5)(1.5,1.5)

\psline(0, 0)(1.5,1.5)
\psline(0, 1.5)(1.5,0)

\psline(1.5, 1.5)(3,1.5)
\psline(1.5, 0)(3,0)

\psline(0,0)(-1.5,0)
\psline(0,1.5)(-1.5,1.5)
}
\end{pspicture}
\psset{unit=0.7cm}
\begin{pspicture}(-4.5, -1)(2,1)
\rput(0.5,-1){
\rput(0,0){$\bullet$}
\rput(0,1.5){$\bullet$}
\rput(1.5,0){$\bullet$}
\rput(1.5,1.5){$\bullet$}
\rput(3, 1.5){$\bullet$}
\rput(3, 0){$\bullet$}
\rput(-1.5, .75){$\bullet$}

\psline(0,0)(0,1.5)
\psline(0,0)(1.5,0)
\psline(1.5,0)(1.5,1.5)
\psline(0, 1.5)(1.5,1.5)

\psline(0, 0)(1.5,1.5)
\psline(0, 1.5)(1.5,0)

\psline(1.5, 1.5)(3,1.5)
\psline(1.5, 0)(3,0)

\psline(0,0)(-1.5,.75)
\psline(0,1.5)(-1.5,.75)
}
\end{pspicture}

\psset{unit=0.7cm}
\begin{pspicture}(-.5,-1)(3,2)
\rput(-.3,-1){
\rput(0,0){$\bullet$}
\rput(0,1.5){$\bullet$}
\rput(1.5,0){$\bullet$}
\rput(1.5,1.5){$\bullet$}
\rput(3, .75){$\bullet$}
\rput(-1.5, .75){$\bullet$}

\psline(0,0)(0,1.5)
\psline(0,0)(1.5,0)
\psline(1.5,0)(1.5,1.5)
\psline(0, 1.5)(1.5,1.5)

\psline(0, 0)(1.5,1.5)
\psline(0, 1.5)(1.5,0)

\psline(1.5, 1.5)(3,.75)
\psline(1.5, 0)(3,.75)

\psline(0,0)(-1.5,.75)
\psline(0,1.5)(-1.5,.75)
}
\end{pspicture}
\psset{unit=0.7cm}
\begin{pspicture}(-2,-1)(2,1)
\rput(-1,-0.27){
\rput(0,0){$\bullet$}
\rput(1.5,0){$\bullet$}
\rput(2.56,1.05){$\bullet$}
\rput(2.56,-1.05){$\bullet$}
\rput(3.62, 0){$\bullet$}
\rput(5.12,0){$\bullet$}

\psline(0,0)(1.5,0)
\psline(1.5,0)(3.62,0)
\psline(1.5,0)(2.56,1.05)
\psline(1.5, 0)(2.56,-1.05)
\psline(2.56, 1.05)(2.56,-1.05)
\psline(2.56, 1.05)(3.62,0)
\psline(2.56, -1.05)(3.62,0)
\psline(2.56, 1.05)(5.12,0)
\psline(2.56, -1.05)(5.12,0)
\psline(5.12, 0)(3.62,0)
}
\end{pspicture}
\end{center}
\caption{}\label{triangle}
\end{figure}

\newpage

\section{Generic  Bi-CM graphs}
\label{generic}

As we have already seen in the first section, the Alexander dual $J=I^\vee_G$ of the edge ideal of a bi-CM graph $G$ is a Cohen--Macaulay ideal of codimension 2 with linear resolution. The ideal $J$ may have several distinct relation matrices with respect to the unique minimal monomial set of generators of $J$. As shown in \cite{BH1}, one may attach to each of the  relation matrices  of $J$  a tree as follows:  let $u_1,\ldots,u_m$ be the unique minimal set of  generators of $J$. Let $A$ be one of the relation matrices of $J$.  Because $J$ has a linear resolution, the generating relations of $J$ may be chosen  all of the form $x_ku_i-x_lu_j=0$. This implies that in each row of the $(m-1)\times m$-relation matrix $A$ there are exactly two non-zero entries (which are variables with different signs). We call such relations, {\em relations of binomial type}.

\begin{Example}
\label{relation}
Consider the bi-CM graph  $G$ on the vertex set $[5]$ and edges $\{1,2\}$ $\{2,3\}$, $\{3,1\}$, $\{2,4\}$, $\{3,4\}$, $\{4,5\}$ as displayed in Figure~\ref{a}.

\begin{figure}[hbt]
\begin{center}
\psset{unit=0.7cm}
\begin{pspicture}(-2,-3)(3,1)
\rput(-0.8,-1){
\rput(0,0){$\bullet$}
\rput(1.5,1){$\bullet$}
\rput(1.5,-1){$\bullet$}
\rput(3,0){$\bullet$}
\rput(4.5,0){$\bullet$}

\psline(0,0)(1.5,1)
\psline(0,0)(1.5,-1)
\psline(1.5,1)(1.5,-1)
\psline(1.5,1)(3,0)
\psline(1.5,-1)(3,0)
\psline(3,0)(4.5,0)

\rput(-0.5,0){$x_{1}$}
\rput(1.5,1.4){$x_{2}$}
\rput(1.5,-1.4){$x_{3}$}
\rput(3,0.5){$x_{4}$}
\rput(4.5, 0.5){$x_{5}$}
}
\end{pspicture}
\end{center}
\caption{}
\label{a}
\end{figure}

The ideal $J=I_G^\vee$ is generated by $u_1=x_2x_3x_4$, $u_2=x_1x_3x_4$, $u_3=x_2x_3x_5$ and $u_4=x_1x_2x_4$. The relation matrices with respect to $u_1,u_2,u_3$ and $u_4$ are the matrices
\[
A_1=
\begin{pmatrix}
 x_1 & -x_2 & 0 & 0 \\
x_5 & 0 & -x_4 & 0 \\
x_1 & 0 & 0 & -x_3
\end{pmatrix},
\]
and
\[
A_2=
\begin{pmatrix}
 x_1 &-x_2 & 0 & 0 \\
x_5 & 0 &-x_4 & 0 \\
0 & x_2 & 0 & -x_3
\end{pmatrix}.
\]
\end{Example}

\medskip
\noindent
Coming back to the general  case,  one assigns to the relation matrix $A$ the following graph $\Gamma$: the vertex set  of $\Gamma$ is the set $V(\Gamma)=\{1,2,\ldots,m\}$, and $\{i,j\}$ is said to be an edge of $\Gamma$ if and only if  some row of $A$ has  non-zero entries for the $i$th- and $j$th-component.
It is remarked in \cite{BH1} and easy to see that $\Gamma$ is a tree. This tree is in general not uniquely determined by $G$.

\medskip
In our Example~\ref{relation} the relation tree of $A_1$ is

\begin{figure}[hbt]
\begin{center}
\psset{unit=0.8cm}
\begin{pspicture}(-1,-2)(4,1.5)
\rput(-0.8,-1){
\rput(0,0){$\bullet$}
\rput(1.5,0){$\bullet$}
\rput(3,0){$\bullet$}
\rput(1.5,1.4){$\bullet$}

\psline(0,0)(1.5,0)
\psline(1.5,0)(3,0)
\psline(1.5,0)(1.5,1.4)

\rput(0,-0.5){$x_{4}$}
\rput(1.5,-0.5){$x_{1}$}
\rput(3,-0.5){$x_{2}$}
\rput(1.5, 1.9){$x_{3}$}
}
\end{pspicture}
\end{center}
\caption{}\label{b}
\end{figure}

while the relation tree of $A_2$ is

\begin{figure}[hbt]
\begin{center}
\psset{unit=0.8cm}
\begin{pspicture}(0,-2)(3,0)
\rput(-0.8,-1){
\rput(0,0){$\bullet$}
\rput(1.5,0){$\bullet$}
\rput(3,0){$\bullet$}
\rput(4.5,0){$\bullet$}

\psline(0,0)(1.5,0)
\psline(1.5,0)(3,0)
\psline(3,0)(4.5,0)

\rput(0,-0.5){$x_{3}$}
\rput(1.5,-0.5){$x_{1}$}
\rput(3,-0.5){$x_{2}$}
\rput(4.5,-0.5){$x_{4}$}
}
\end{pspicture}
\end{center}
\caption{}\label{c}
\end{figure}

Now let $J$ be any codimension 2 Cohen-Macaulay monomial ideal with linear resolution. Then, as observed in Section~\ref{properties},  $J^\vee=I_G$  where $G$ is a bi-CM graph.  Now we follow Naeem \cite{N} and define for any given tree $T$ on the vertex set $[m]=\{1,\ldots,m\}$ with edges $e_1,\ldots,e_{m-1}$ the $(m-1)\times m$-matrix $A_T$ whose entries $a_{kl}$ are defined as follows: we  assign to the $k$th edge $e_k=\{i,j\}$ of $T$ with $i<j$ the
$k$th row of $A_T$ by setting
\begin{eqnarray}
\label{genericmatrix}
a_{kl}= \left\{\begin{array}{ll} x_{ij},&\text{ if $l =i$,}\\
-x_{ji},&\mbox{ if $l=j$,}\\
0,&\mbox{ otherwise.}
\end{array}\right.
\end{eqnarray}
The matrix $A_T$ is called the {\em generic matrix}  attached to the tree $T$.

\medskip
By the Hilbert-Burch theorem \cite{BH},  the matrix $A_T$ is the relation matrix of the ideal $J_T$  of   maximal minors of $A_T$, and  $J_T$ is a Cohen-Macaulay ideal of codimension $2$ with linear resolution.

We let $G_T$ be the graph such that $I_{G_T}=J^\vee$, and call $G_T$ the {\em generic bi-CM graph} attached to $T$.

Our discussion so far yields

\begin{Proposition}
\label{genericbicm}
For any tree $T$, the graph $G_T$ is bi-CM.
\end{Proposition}

In order to describe the vertices and edges  of $G_T$, let $i$ and $j$ be any two vertices of the tree $T$. There exists a unique path $P: i=i_0,i_1,\ldots,i_r=j$ from $i$ to $j$. We set $b(i,j)=i_1$ and call $b(i,j)$ the {\em begin} of $P$, and set $e(i,j)=i_{r-1}$ and call $e(i,j)$ the {\em end} of $P$.

It follows from  \cite[Proposition 1.4]{N} that $I_{G_T}$ is generated by the monomials $x_{ib(i,j)}x_{je(i,j)}$. Thus the vertex set  of the graph $G_T$ is given as
\[
V(G_T)=\{(i,j), (j,i)\: \text{$\{i,j\}$ is an edge of $T$}\}.
\]
In particular, $\{(i,k), (j,l)\}$ is an edge of $G_T$ if and only if there exists a path $P$ from $i$ to $j$ such that  $k=b(i,j)$ and $l=e(i,j)$.

\medskip
In Example~\ref{relation}, let $T_1$ and $T_2$ be the relation trees of $A_1$ and $A_2$, respectively. Then the generic matrices corresponding to these trees are
\medskip
\[
B_1=
\begin{pmatrix}
 x_{12} & -x_{21} & 0 & 0 \\
x_{13} & 0 & -x_{31} & 0 \\
x_{14} & 0 & 0 & -x_{41}
\end{pmatrix},
\]
and
\[
B_2=
\begin{pmatrix}
 x_{12} &-x_{21} & 0 & 0 \\
x_{13} & 0 &-x_{31} & 0 \\
0 & x_{24} & 0 & -x_{42}
\end{pmatrix}.
\]

\medskip
The generic graphs corresponding to the trees $T_1$ and $T_2$ are displayed in Figure~\ref{d}.
\begin{figure}[hbt]
\begin{center}
\psset{unit=0.7cm}
\begin{pspicture}(-.2,-3)(2,2)
\rput(-0.8,-1){
\rput(0,0){$\bullet$}
\rput(1.5, 0){$\bullet$}
\rput(.75,1.3){$\bullet$}
\rput(.75,2.6){$\bullet$}
\rput(-1.1,-.7){$\bullet$}
\rput(2.6,-.7){$\bullet$}

\rput(.75,-2.5){$G_{T_1}$}

\psline(0,0)(1.5, 0)
\psline(0,0)(.75,1.3)
\psline(.75, 1.3)(1.5, 0)
\psline(.75,1.3)(.75,2.6)

\psline(0,0)(-1.1,-.7)
\psline(1.5,0)(2.6, -.7)

\rput(-1.2, -1.2){$x_{41}$}
\rput(-0.6,.2){$x_{14}$}
\rput(2.1,.1){$x_{31}$}
\rput(2.7,-1.2) {$x_{13}$}
\rput(1.35, 1.3){$x_{21}$}
\rput(1.35,2.6){$x_{12}$}
}
\end{pspicture}
\psset{unit=0.7cm}
\begin{pspicture}(-5,-3)(2,2)
\rput(-0.8,-1){
\rput(.37, .66){$\bullet$}
\rput(1.87, .66){$\bullet$}
\rput(1.12,1.96){$\bullet$}
\rput(1.12, -.66){$\bullet$}
\rput(3.37, .66){$\bullet$}

\rput(1.12,-2.5){$G_{T_2}$}

\psline(.37, .66)(1.87, .66)
\psline(.37,.66)(1.12,1.96)
\psline(.37,.66)(1.12, -.66)
\psline(1.12, 1.96)(1.87, .66)
\psline(1.12, -.66)(1.87, .66)
\psline(1.87,.66)(3.37, .66)
\rput(-.23, .66){$x_{12}$}
\rput(1.12,2.46){$x_{21}$}
\rput(1.12, -1.16){$x_{42}$}
\rput(2.1, 1.16){$x_{31}$}
\rput(3.95, .66){$x_{13}$}
}
\end{pspicture}
\end{center}
\caption{}\label{d}
\end{figure}

\medskip
It follows from this description that  $G_T$ has $2(m-1)$ vertices. Since $G_T$ is bi-CM, the number of edges of $G_T$ is ${n-c+1\choose 2}$, see Corollary~\ref{edges}.  Here $n-c$ is the degree of the generators of $I_G^\vee$ which is $m-1$. Hence $G_T$ has ${m\choose 2}$ edges. Among the edges of $G_T$ are in particular the $m-1$ edges $\{(i,j),(j,i)\}$ where $\{i,j\}$ is an edge of $T$.

\begin{Proposition}
\label{pairwise}
Let $A$ be the relation matrix of a codimension $2$ Cohen-Macaulay monomial ideal $J$ with linear resolution, and assume that all the variables appearing in $A$ are pairwise distinct. Let $T$ be the relation tree of $A$.  Then $J$ is isomorphic to $J_T$ and $J$ admits the unique relation tree, namely $T$ .
\end{Proposition}

\begin{proof}
Since all  variables  appearing in $A$ are pairwise distinct, we may rename the variables  appearing in a binomial type relation  and call them  as in the generic matrix $x_{ij}$ and $x_{ji}$. Then $A$ becomes $A_T$ and this  shows that $J\iso J_T$.

To prove the uniqueness of the relation tree, we first notice that the  shifts in the multigraded free resolution of $J$ are uniquely determined and independent of the particular choice of the relation matrix  $A$. A possibly different relation matrix $A'$ can arise from $A$ only be row operations with rows of the same multidegree. Let $r_1,\ldots,r_l$ by rows of $A$ with the same multidegree corresponding to binomial type relations,  and fix a column $j$. Then the non-zero $j$th columns of each of the $r_i$ must be the same, up to a sign. Since we assume that the variables appearing in $A$ are pairwise distinct, it follows that $l=1$. In particular, there is, up to the order of the rows,  only  one relation matrix with rows corresponding to binomial type relations. This shows that $T$ is uniquely determined.
\end{proof}

\section{Inseparable models of Bi-CM graphs}
\label{main}

In order to state the main result of this paper we recall the concept of inseparability introduced by Fl{\o}ystad et al in \cite{FGH}, see also \cite{L}.

Let $S=K[x_1,\ldots,x_n]$ be the polynomial ring over the field $K$ and  $I\subset S$  a squarefree  monomial ideal minimally generated by the monomials $u_1,\ldots, u_m$. Let $y$ be an indeterminate over $S$. A monomial  ideal $J\subset S[y]$ is called a {\em separation} of $I$ for the variable $x_i$ if  the following holds:
\begin{enumerate}
\item[(i)] the ideal $I$ is the image of $J$ under  the $K$-algebra homomorphism $S[y]\to S$ with $y\mapsto x_i$ and $x_j\mapsto x_j $ for all $j$;

\item[(ii)] $x_i$ as well as $y$ divide some minimal generator of $J$;

\item[(iii)] $y-x_i$ is a non-zero divisor of $S[y]/J$.
\end{enumerate}
The ideal $I$ is called {\em separable}  if it admits a separation, otherwise {\em inseparable}.  If $J$ is an  ideal which is obtained from $I$ by a finite number of separation steps, then we say that $J$ {\em specializes} to $I$. If moreover, $J$ is inseparable, then $J$ is called an {\em inseparable model} of $I$.  Each monomial ideal admits an inseparable model, but in general not only one. For example, the separable models of the powers of the graded maximal ideal of $S$ have been considered by Lohne \cite{L}.

\medskip
Forming the Alexander dual behaves well with respect to specialization and separation.

\begin{Proposition}
\label{alexander}
Let $I\subset S$ be a squarefree  monomial ideal. Then the following holds:
\begin{enumerate}
\item[{\em (a)}] If $J$ specializes to $I$, then $J^\vee$ specializes to $I^\vee$.

\item[{\em (b)}] The ideal $I$ is separable if and only $I^\vee$ is separable.
\end{enumerate}
\end{Proposition}

\begin{proof}
(a) It follows from \cite[Proposition 7.2]{FGH} that if $L\subset S[y]$ is a monomial ideal such that $y-x_i$ is a regular element on $S[y]/L$ with  $(S[y]/L)/(y-x_i)(S[y]/L)\iso S/I$, then $y-x_i$ is a regular element on $S[y]/L^\vee$ and $(S[y]/L^\vee)/(y-x_i)(S[y]/L^\vee)\iso S/I^\vee$. Repeated applications of this fact  yields the desired result.

(b) We may assume  that the ideal $L$  as in (a) is a separation of $I$ with respect to $x_i$. Since (a) holds,  it remains to show that  $y$ as well as $x_i$ divides some generator of $L^\vee$. By assumption this is the case for $L$.  Suppose that $y$ does not divide any generator of $L^\vee$. Then it follows from the definition of the Alexander dual that $y$ also does not divide  any generator of $(L^\vee)^\vee$. This is a contradiction, since $L=(L^\vee)^\vee$. Similarly it follows that $x_i$ divides some generator of  $L^\vee$.
\end{proof}

\medskip
We now apply these concepts to edge ideals. Let $G$ be a  graph on the vertex set $[n]$. We call $G$ {\em separable} if $I_G$ is separable, and otherwise {\em inseparable}. Let $J$ be a separation of $I_G$ for the variable $x_i$. Then by the definition of separation,  $J$  is again an edge ideal, say $J=I_{G'}$ where $G'$ is a graph with one more vertex than $G$. The graph  $G$ is obtained from $G'$ by identifying this new vertex with the vertex $i$ of $G$.
Algebraically, this identification amounts to say that $S/I_G\iso (S'/I_{G'})/(y-x_i)(S'/I_{G'})$, where $S'=S[y]$ and $y-x_i$ is a non-zerodivisor of $S'/I_{G'}$. In particular, it follows that  $I_G$ and $I_{G'}$ have the same graded Betti-numbers. In other words, all important homological  invariants of $I_G$ and $I_{G'}$ are the same. It is therefore of interest to classify all inseparable graphs. An attempt for this classification is given in \cite{ABHL}.

\begin{Example}
\label{separable}
Let  $G$ be the triangle and $G'$ be the line graph displayed in Figure~\ref{triangle}.

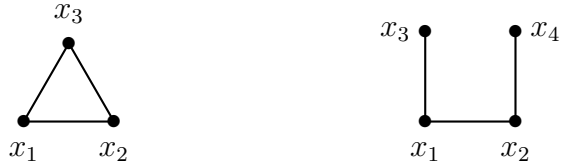
\begin{figure}[hbt]
\begin{center}
\psset{unit=0.8cm}
\begin{pspicture}(-0.2,-2)(2,1)
\rput(-0.8,-1){
\rput(0,0){$\bullet$}
\rput(1.5,0){$\bullet$}
\rput(.75,1.3){$\bullet$}

\psline(0,0)(.75,1.3)
\psline(0,0)(1.5,0)
\psline(.75,1.3)(1.5,0)

\rput(0,-0.5){$x_{1}$}
\rput(.75,1.8){$x_{3}$}
\rput(1.5,-0.5){$x_{2}$}
}
\end{pspicture}
\psset{unit=0.8cm}
\begin{pspicture}(-4.5,-2)(2,1)
\rput(-0.8,-1){
\rput(0,0){$\bullet$}
\rput(0,1.5){$\bullet$}
\rput(1.5,0){$\bullet$}
\rput(1.5,1.5){$\bullet$}

\psline(0,0)(0,1.5)
\psline(0,0)(1.5,0)
\psline(1.5,0)(1.5,1.5)

\rput(0,-0.5){$x_{1}$}
\rput(-0.5,1.5){$x_{3}$}
\rput(1.5,-0.5){$x_{2}$}
\rput(2,1.5){$x_{4}$}
}
\end{pspicture}
\end{center}
\caption{A triangle and its inseparable model}\label{triangle}
\end{figure}

Then  $I_{G'}=(x_1x_2, x_1x_3, x_2x_4).$ Since $\Ass(I_{G'})=\{ (x_1, x_2), (x_1, x_4), (x_2, x_3)\}$, it follows that $x_3-x_4$ is a non-zero divisor on $S'/I_{G'}$ where $S'=K[x_1, x_2, x_3,x_4]$.  Moreover,  $(S'/I_{G'})/(x_3-x_4)(S'/I_{G'})\iso S/I_G$. Therefore, the triangle in Figure~\ref{triangle} is obtained as a specialization from the line graph in Figure~\ref{triangle} by identifying the vertices $x_3$ and $x_4$.

\end{Example}
\medskip
\noindent

We denote by $G^{(i)}$ the complementary graph of the restriction $G_{N(i)}$ of $G$ to $N(i)$ where $N(i)=\{j\:\; \{j,i\}\in E(G)\}$ is the neighborhood of $i$. In other words, $V(G^{(i)})=N(i)$ and $E(G^{(i)})=\{\{j,k\}\:\; j,k\in N(i) \text{ and } \{j,k\}\not\in E(G)\}$.  Note that $G^{(i)}$ is disconnected if and only if  $N(i)=A\cup B$, where $A,B\neq \emptyset$, $A\sect B=\emptyset$ and all  vertices of $A$ are adjacent to those of $B$.

Here we will need the following result of \cite[Theorem 3.1]{ABHL}.

\begin{Theorem}
\label{combinatorially}
 The following conditions are equivalent:
\begin{enumerate}
\item[{\em (a)}] The  graph $G$ is inseparable;

\item[{\em (b)}] $G^{(i)}$ is connected for all $i$.
\end{enumerate}
\end{Theorem}

Now we are ready to state our main result.

\begin{Theorem}
\label{maintheorem}
{\em (a)} Let $T$ be a tree. Then $G_T$ is an  inseparable bi-CM graph.

{\em (b)} For any inseparable bi-CM  graph $G$, there exists a unique tree $T$ such that $G\iso G_T$.

{\em (c)} Let $G$ be any bi-CM graph. Then there exists a  tree $T$ such that $G_T$ is an inseparable model of $G$.
\end{Theorem}

\begin{proof} (a) By Corollary~\ref{genericbicm}, $G_T$ is a bi-CM graph.  In order to see that $G_T$ is inseparable we apply the criterion given in Theorem~\ref{combinatorially},  and thus we have to prove that for each vertex  $(i,j)$ of $G_T$ and for each disjoint union  $N((i,j)) =A\union B$ of the neighborhood of $(i,j)$ for which  $A\neq \emptyset \neq B$, not all   vertices of $A$ are adjacent to those of $B$.

As follows from the discussion  in Section~\ref{generic},
\[
N((i,j)) =\{(k,l)\: \text{there exists a path from $i$  to $l$, and $j=b(i,l)$ and $k=e(i,l)$}\}.
\]
In particular, $(j,i)\in N((i,j))$. Let  $N((i,j)) =A\union B$, as above. We may assume that $(j,i)\in A$. Since $T$ is a tree, then there is no path from $j$ to any  $l$ with $(k,l)\in N((i,j))$, because otherwise we would have a loop in $T$.  This shows that $(j,i)$ is connected to no vertex in $B$, as desired.

(b) Let $A$ be a relation matrix of $J=I_G^\vee$ and $T$ the relation tree of $A$. The non-zero entries  of $A$ are variables with sign $\pm 1$. Say the $k$th row of $A$ has the non-zero entries $a_{ki_k}$ and $a_{kj_k}$ with $i_k<j_k$.  We may assume that the variable representing $a_{ki_k}$ has a positive sign while that $a_{kj_k}$ has a negative sign, and that this is so for each row.  We claim that the variables appearing in the non-zero entries of $A$ are pairwise distinct. By Proposition~\ref{pairwise} this then implies that $T$ is the only relation tree of $J$ and that $G\iso G_T$.

In order to prove the claim, we consider the generic matrix $A_T$  corresponding to $T$. Let $S'$ be the polynomial ring over $S$ in the variables $x_{ij}$ and $x_{ji}$ with $\{i,j\}\in E(T)$. For each $k$ we consider the linear forms $\ell_{k1}=x_{i_kj_k}- a_{ki_k}$ and $\ell_{k2} =x_{j_ki_k}- a_{kj_k}$.
For example, for the matrix $A_2$ in Example~\ref{relation} the linear forms are
$\ell_{11}=x_{12}-x_1$, $\ell_{12}=x_{21}-x_2$, $\ell_{21}=x_{13}-x_5$, $\ell_{22}=x_{31}-x_4$, $\ell_{31}=x_{24}-x_2$ and $\ell_{32}=x_{42}-x_3$.

We let $\ell$ be the sequence of linear form $\ell_{11},\ell_{12},\ldots, \ell_{m-1,1},\ell_{m-1,2}$ in $S'$. Then $(S'/J_TS')/(\ell)(S'/J_TS')\iso S/J$. Since both ideals, $J$ as well as $J_T$,  are Cohen-Macaulay ideals of codimension $2$, it follows that  $\ell$ is a regular sequence on $S'/J_TS'$. Thus, assuming  the variables appearing in the non-zero entries of $A$ are not all pairwise distinct, we see that $J$ is separable.
Indeed, suppose that the variable $x_k$ appears at least twice in the matrix. Then we replace only one of the $x_k$ by the corresponding generic variable $x_{ij}$ to obtain the matrix $A'$. Let $J'$ be the ideal  of maximal minors of $A'$. It follows from the above discussions that $x_{ij}-x_k$ is a regular element of $S[x_{ij}]/J'$. In order to see that $J'$ is a separation of $J$ it remains to be shown that $x_{ij}$  as well as $x_k$ appear as factors of generators of $J'$.  Note that $J'$ is a specialization of $J_T$. The minors of $A_T$ which are the generators of $J_T$ are the monomials $\prod_{i=1\atop i\neq j}^{m+1}x_{ib(i,j)}$ for $j=1,\ldots,m+1$, see \cite[Proposition 1.2]{N}. From this description of the generators of $J_T$  it follows that all entries of $A_T$ appear as factors of  generators of $J_T$. Since $J'$ is a specialization of $J_T$, the same holds true for $J'$, and since $x_{ij}$ as well as $x_k$ are entries of $A'$, the desired conclusion follows.

Now since we know that $J$ is separable, Proposition~\ref{alexander}(b) implies  that $G$ is separable as well. This is  a contradiction.

(c) Let $A$ be a relation matrix of $J=I_G^\vee$ and $T$  the corresponding relation tree. As shown in the proof of part (b), $J_T$ specializes to $J$, and hence $I_{G_T}$ specializes to $I_G$, by Proposition~\ref{alexander}(a). By part (a), the graph   $G_T$ is inseparable. Thus  we conclude that $G_T$ is an inseparable model of  $G$, as desired.
\end{proof}

\end{document}